\documentclass[reqno]{amsart}
\usepackage{amsmath, amssymb, latexsym, amsthm, amsfonts, mathrsfs,
  amscd, color, bbm}
\usepackage{scalerel, stackengine, microtype}
\usepackage{tikz-cd}
\stackMath
\newcommand\reallywidehat[1]{%
\savestack{\tmpbox}{\stretchto{%
  \scaleto{%
    \scalerel*[\widthof{\ensuremath{#1}}]{\kern-.6pt\bigwedge\kern-.6pt}%
    {\rule[-\textheight/2]{1ex}{\textheight}}
  }{\textheight}%
}{0.5ex}}%
\stackon[1pt]{#1}{\tmpbox}%
}
\usepackage{setspace}

\renewcommand{\eqref}[1]{(\ref{#1})}   

\numberwithin{equation}{section}

\theoremstyle{plain}
\newtheorem{theorem}{Theorem}[section]
\newtheorem{lemma}[theorem]{Lemma}
\newtheorem{corollary}[theorem]{Corollary}
\newtheorem{proposition}[theorem]{Proposition}
\newtheorem{conjecture}[theorem]{Conjecture}

\theoremstyle{definition}
\newtheorem{definition}[theorem]{Definition}
\newtheorem{remark}[theorem]{Remark}

\theoremstyle{definition}

\newcommand{\Q}{{\mathbbm Q}}

\newcommand{\Z}{{\mathbbm Z}}
\newcommand{\C}{{\mathbbm C}}
\newcommand{\N}{{\mathbbm N}}
\newcommand{\F}{{\mathbbm F}}

\newcommand{\ie}{\;\textit{i}.\textit{e}.\;\;} 

\title[irreducibility of eventually $2$-periodic curves]{Irreducibility of eventually $2$-periodic curves in the moduli space of cubic polynomials}
\author{Niladri Patra}
\address{School of Mathematics, Tata Institute of Fundamental Research, Homi Bhabha Road, Navy Nagar, Colaba, Mumbai-400005}
\email{niladript@gmail.com, niladri@math.tifr.res.in}


\begin{document}

\maketitle

\begin{abstract}
    Consider the moduli space, $\mathcal{M}_{3},$ of cubic polynomials over $\C$, with a marked critical point. Let $\mathscr{S}_{k,n}$ be the set of all points in $\mathcal{M}_{3}$ for which the marked critical point is strictly $(k,n)$-preperiodic. Milnor conjectured that the affine algebraic curves $\mathscr{S}_{k,n}$ are irreducible, for all $k \geq 0, n>0$. In this article, we show the irreducibility of eventually $2$-periodic curves, i.e. $\mathscr{S}_{k,2},\; k\geq 0$ curves. We also note that the curves, $\mathscr{S}_{k,2},\; k\geq 0$, exhibit a possible splitting-merging phenomenon that has not been observed in earlier studies of $\mathscr{S}_{k,n}$ curves. Finally, using the irreducibility of $\mathscr{S}_{k,2}$ curves, we give a new and short proof of Galois conjugacy of unicritical points lying on $\mathscr{S}_{k,2}$, for even natural number $k$.   
\end{abstract}

\section{Introduction}
    Complex dynamics (also called holomorphic dynamics) is the study of self-iterates of rational functions on the Riemann sphere. Spaces of rational functions (in particular, polynomials) of fixed degree have been an integral part of complex dynamics since $1980s$ (\cite{BM81}, \cite{BH88}, \cite{M90}, \cite{BDK91}, \cite{BH92}, \cite{MP92}, \cite{R06}, \cite{LQ07}, \cite{inbook}, \cite{BKM10}, \cite{Milnor+2014+15+24}). It began with the moduli space of quadratic polynomials and its subsets of dynamical nature, like the Mandelbrot set. In two consecutive papers (\cite{BH88}, \cite{BH92}), Branner and Hubbard pioneered the study of the moduli space of cubic polynomials. Later in a $1991$ preprint, published as \cite{inbook}, Milnor restructured this moduli space by considering cubic polynomials with a \emph{marked} critical point. This space is a ramified $2$ to $1$ cover of the moduli space of cubic polynomials. In this article, we will partially answer a question, raised by Milnor (\cite{inbook}), about the irreducibility of some dynamically interesting affine algebraic curves in the moduli space of cubic polynomials with a marked critical point.
    
    Let $f$ be a polynomial over $\C$. For any integer $m\geq 0$, we denote the iteration of $f$ with itself $m$ times, as $f^{m}$, $\ie f^{0}= Id,\; f^{m}= f^{m-1} \circ f,\text{ for all } m\in \N$. For $x\in \C$, the \textit{forward orbit} of $x$ is defined to be the set of images of $x$ under all the iterates of $f$, $\ie \{f^{m}(x)\; |\;m \geq 0\}$.

    A point $x\in \C$ is called a \emph{periodic point} of period $n$ if $f^{n}(x)=x$. It is called \emph{strictly} $n$-\emph{periodic point} if $n$ is the smallest positive integer for which $f^{n}(x)=x$. A point $x\in \C$ is called a $(k,n)$-\emph{preperiodic point} if $f^{k}(x)$ is a periodic point of period $n$, $\ie f^{k+n}(x)=f^{k}(x)$. It is called \emph{strictly} $(k,n)$-\emph{preperiodic point} if $f^{k}(x)$ is strictly $n$-periodic and $f^{l}(x)$ is not periodic for any $0\leq l < k$. 

    For a polynomial $f\in \C[z]$, the roots of the derivative of $f$ are called the \emph{finite critical points} of $f$. Let us consider the set $S_{3}$ of all cubic polynomials over $\C$, with a marked (finite) critical point. Two polynomials that are affine conjugate to each other, exhibit the same dynamical behaviour. So, we consider the quotient space of $S_{3}$, by identifying polynomials that are affine conjugate to each other, and the affine conjugation map sends the marked critical point of the first to the marked critical point of the latter. This space, $\mathcal{M}_{3}$, is called \emph{the moduli space of cubic polynomials with a marked critical point}. In other words, $\mathcal{M}_{3}$ is the space 
    $$\{(f,a)\;|\; f \in S_{3}, a \text{ is a finite critical point of } f\}/ \sim ,$$ 
    where the equivalence relation $\sim$ is given by $(f,a) \sim (g,a')$ iff there is $h\in Aut(\C)$ such that $h\circ f\circ h^{-1}= g$ and $h(a)=a'$.

    A polynomial in $\C[z]$ is called \emph{monic} if its leading coefficient is one, and called \emph{reduced} (or, \emph{centered}) if the sum of its roots is zero. Observe that, any polynomial is affine conjugate to a monic, reduced polynomial. Hence, $\mathcal{M}_{3}$ can be seen as set of affine conjugacy classes of monic, reduced, cubic polynomials over $\C$, with a marked critical point. From \cite{inbook}, every monic, reduced cubic polynomial, with a marked critical point, can be written in the \emph{modified Branner-Hubbard normal form} as,
    \begin{equation}\label{maineq}
        f_{a,b}(z)=z^{3}- 3a^{2}z +2a^{3} +b,
    \end{equation} 
    with $\pm a$ as its finite critical points, $a$ is the marked critical point and $f(a)=b$ is a finite critical value. Let $a,b,a',b' \in \C$. Brief calculation shows that, $f_{a,b}$ and $f_{a',b'}$ are affine conjugate to each other iff either $(a,b)=(a',b')$ or $(a,b)=(-a',-b')$. Hence, the moduli space $\mathcal{M}_{3}$ can be identified as, 
    $$\mathcal{M}_{3} \longleftrightarrow \C^{2}/\left((a,b)\sim(-a,-b)\right).$$ 
    The space $\C^{2}/\left((a,b)\sim(-a,-b)\right)$ is isomorphic to the image of $\C^{2}$ under the affine Veronese map, $\; \C^{2} \rightarrow \C^{3},\; (a,b) \mapsto (a^{2}, ab, b^{2})$. Hence, $\mathcal{M}_{3} = \C^{2}/\left((a,b)\sim(-a,-b)\right)$ is a $2$-dimensional affine variety.

    Let $k \geq 0$ and $n >0$ be integers. Consider the set, $\mathscr{S}_{k,n}$, of all points $(a,b)\in \mathcal{M}_{3}$ such that the marked critical point $a$ is strictly $(k,n)$-preperiodic under the polynomial map $f_{a,b}$. The set $\mathscr{S}_{k,n}$ is an affine algebraic curve in $\mathcal{M}_{3}$. Milnor \cite{inbook} conjectured that the $\mathscr{S}_{0,n}, n\in \N$ curves are all irreducible. In general, it is conjectured that,
    \begin{conjecture}
        For any choice of integers $k\geq 0$ and $n>0$, the curve $\mathscr{S}_{k,n}$ is irreducible.
    \end{conjecture}
    Buff, Epstein, and Koch (\cite{article1}) proved this conjecture for $\mathscr{S}_{k,1}$ curves. Arfeux and Kiwi (\cite{AK20}) have shown that $\mathscr{S}_{0,n}$ curves are irreducible, for any $n \in \N$. In this article, we will prove this conjecture for $\mathscr{S}_{k,2}$ curves, for any non-negative integer $k$ (see Theorem \ref{irrck23}). We state the theorem below.
    \begin{theorem}\label{main1}
        For any non-negative integer $k$, the curve $\mathscr{S}_{k,2}$ is irreducible.
    \end{theorem}
    Our proof of Theorem \ref{main1} is of arithmetic nature and mimics the approach taken for the unicritical case in \cite{article1}. Let $k\geq 0, n>0$. We form polynomials $h_{k,n}\in \Z[a,b]$ such that $\mathscr{S}_{k,n}$ is Zariski dense in the curve of $h_{k,n}$ in $\mathcal{M}_{3}$, with finite complement. We form the polynomials $h_{0,2}$ and $h_{1,2}$ explicitly and show that they are irreducible over $\C$. For $k\geq 2$, We show that $h_{k,2}$ polynomials are generalised $3$-Eisenstein polynomials with respect to $h_{1,2}$. Hence, they are irreducible over $\Q$. If the polynomials $h_{k,2},\; k\geq 0$ are irreducible over $\C$, then we are done. Analogs of these polynomials in the $(k,1),\; k\geq 0$ case (studied in \cite{article1}) and $(0,n),\; n\in \N$ case (studied in \cite{AK20}) has turned out to be irreducible. But in the eventually $2$-periodic case, we observe a possible \emph{splitting-merging phenomenon}. The polynomial $h_{1,2}$ is \emph{reducible}. In general, the polynomials $h_{k,2}$ can be reducible over $\C$. If $h_{k,2}$ is reducible for some $k\geq 0$, then it can split into at most two factors over the field $\Q[i]$. We show that both of these factors, that lie in $\Q[i][a,b]$, have a smooth $\Q[i]$-rational point. Using extension of irreducibility (Corollary \ref{p}), we get that both of these factors are irreducible over $\C$. Moreover, we show that the irreducible curves in $\C^{2}$ corresponding to these two factors merge together under the equivalence relation $(a,b) \sim (-a,-b)$, generating one irreducible curve in $\mathcal{M}_{3}$, which contains $\mathscr{S}_{k,2}$ as a Zariski dense subset. This completes the proof of Theorem \ref{main1}.

    As a consequence of this approach, we find an explicit form of $h_{k,2}$ polynomials. 
    \begin{proposition}
        The polynomial $h_{0,2}$ is $(b-a)(b+2a)+1$. Let $k\in \N$. For $k$ odd,
    \begin{equation}
        h_{k,2} = \dfrac{f_{k,2}\cdot f_{k-1,1}}{h_{0,2}\cdot f_{k-1,2}\cdot f_{k,1}},
    \end{equation}
    and for $k$ even,
    \begin{equation}
        h_{k,2} = \dfrac{f_{k,2}\cdot f_{k-1,1}}{f_{k-1,2}\cdot f_{k,1}}.
    \end{equation}
    \end{proposition}

    A polynomial over $\C$ is called \emph{unicritical} iff all the finite critical points are equal. Assuming $a=0$ in Equation \eqref{maineq}, we get the normal form of a monic, reduced, unicritical cubic polynomial,
    \begin{equation}\label{unipol}
        f_{b}(z)= z^{3} +b.
    \end{equation}
    For some of the recent studies on unicritical points in moduli spaces of polynomials, see \cite{Milnor+2014+15+24}, \cite{HT15}, \cite{article1}, \cite{G20}, \cite{Gok20}, \cite{BGok22}, \cite{BG23}. In \cite{Milnor+2014+15+24}, Milnor conjectured that,
    \begin{conjecture}\label{conj2}
        Let $k\geq 0, n>0$. The finite set of values of $b$ for which the critical point $0$ is strictly $(k,n)$-preperiodic under $f_{b}$, form one Galois orbit under the action of the absolute Galois group of $\Q$.
    \end{conjecture}
    One can form a polynomial $R_{k,n} \in \Z[b]$, whose solution set is the set of all values of $b$ for which $0$ is strictly $(k,n)$-preperiodic. Hence, Milnor's conjecture can be restated as, for any $k\geq 0, n\geq 1$, the polynomial $R_{k,n}$ is either constant or irreducible over $\Q$. These polynomials can be constant polynomials, for example, $R_{1,n}, n\in \N$ are all equal to one (see Remark \ref{coales}). Vefa Goksel (\cite{G20}, \cite{Gok20}) has shown that $R_{k,1},\: R_{k,2},\; k\geq 0$ polynomials are irreducible over $\Q$. Replacing the normal form in Equation \eqref{unipol} with the form $f_{c}(z)= cz^{3}+1,\; c\in \C$, Buff, Epstein and Koch (\cite{article1}) proved a parallel version of Conjecture \ref{conj2} for $(k,1),\; (k,2),\; k\geq 0$ cases. Both of their approaches rely on studying arithmetic properties of $R_{k,n}$ polynomials and concluding that $R_{k,1}$ and $R_{k,2}$ polynomials are Eisenstein polynomials. Here we show that the Eisenstein nature of $R_{k,2}$ polynomials arise from the Eisenstein nature of $h_{k,2}$ polynomials, mentioned above. Thus we obtain a new and short proof of Conjecture \ref{conj2} for $k$ even and $n=2$ (Theorem \ref{uni}). We state the theorem below.
    \begin{theorem}
        For any even $k\in \Z,\; k\geq 0$, the polynomial $R_{k,2}$ is either constant or irreducible over $\Q$.
    \end{theorem}
    In the case of odd $k\in \N$, this approach gets obstructed due to an extra $(b^{2}+1)$ factor appearing in the factorization of $h_{k,2}(0,b)$. Appearance of this factor obscures the exact valuation of the resultant of $R_{k,2}$ and $b^{2}+1$, with respect to the prime $3$.
    
    We provide a sectionwise summary here. In Section \ref{prelim}, we form the polynomials $h_{k,2},\; k\geq 0,$ and show that the curve in $\mathcal{M}_{3}$ corresponding to $h_{k,2}$ contains $\mathscr{S}_{k,2}$ as a Zariski dense subset. In Section \ref{notations}, we fix some notations to be used in the later sections. In Section \ref{lemmas}, we state some lemmas and tools to be used in the proofs of the later sections. In Section \ref{t''}, we prove irreducibility of $\mathscr{S}_{k,2}$ curves. In Section \ref{explicit}, we obtain the explicit form of the polynomials $h_{k,2}$. In Section \ref{unicritical}, we prove irreducibility results in the unicritical cubic case. Finally, in Section \ref{kqcurves}, we show that our method does not extend directly for $(k,q)$ curves, where $q$ is an odd prime number. 
    
\section{Preliminaries}\label{prelim}
    Recall that, for any $k\geq 0, n>0,$ $\mathscr{S}_{k,n}$ denote the set of all points $(a,b)\in \mathcal{M}_{3}$ such that $a$ is strictly $(k,n)$-preperiodic under $f(z)= z^{3}-3a^{2}z+2a^{3}+b$. In this section, we will form polynomials $h_{k,n}$ such that $\mathscr{S}_{k,n}$ is Zariski dense in the curve of $h_{k,n}$ in $\mathcal{M}_{3}$. As $a$ and $b$ serve as parameters for the set of all monic reduced cubic polynomials over $\C$, we will drop the subscripts $a,b$ from $f_{a,b}$ as in Equation \eqref{maineq}.

    Let $k\geq 0, n>0$. Any point $(a,b)\in \mathcal{M}_{3}$ for which $a$ is $(k,n)$-preperiodic must satisfy the equation, 
        $$f_{k,n}:= f^{k+n}(a)-f^{k}(a)=0,$$
    where $f$ is the polynomial,
    \begin{equation}\label{newmaineq}
        f(z)= z^{3} -3a^{2}z +2a^{3} +b.
    \end{equation}
    Observe that $a$ is not necessarily strictly $(k,n)$-preperiodic, for every point $(a,b)$ that satisfies $f_{k,n}$. In section \ref{lemmas}, we will show that for $0\leq l\leq k,\; 1\leq m$ and $ m|n,$ the polynomial $f_{l,m}$ divides $f_{k,n}$ in $\Z[a,b]$. So, we form the polynomial,
    \begin{equation}\label{hkn}
        h_{k,n} = \frac{f_{k,n}}{\prod_{i} g_{i}^{\alpha_{i}}},
    \end{equation}
    where $g_{i}$ varies over all irreducible factors of $f_{l,m}$ in $\Z[a,b]$, for all $0\leq l\leq k, 1\leq m,\; m|n,\; (k,n)\neq (l,m)\in \Z^{2}$, and $\alpha_{i}$ is the highest power of $g_{i}$ that divides $f_{k,n}$.
    \begin{lemma}\label{dense}
        Let $k\geq 0, n>0$. The polynomial $h_{k,n}$ kills $\mathscr{S}_{k,n}$. Moreover, $\mathscr{S}_{k,n}$ is Zariski dense in the curve of $h_{k,n}$ in $\mathcal{M}_{3}$, with finitely many points in its complement.
    \end{lemma}
    \begin{proof}
        Any point in $\mathscr{S}_{k,n}$ lies on the curve of $f_{k,n}$ but not on the curve of $g_{i}$, for any $g_{i}$ appearing in Equation \eqref{hkn}. Hence, $\mathscr{S}_{k,n}$ is a subset of the curve of $h_{k,n}$ in $\mathcal{M}_{3}$. The complement of $\mathscr{S}_{k,n}$ in the curve of $h_{k,n}$ consists of points $(a,b)$ for which $a$ is $(k,n)$-preperiodic but not strictly. So, each of them is a solution of some $g_{i}$ appearing in Equation \eqref{hkn}. By definition of $h_{k,n}$, $h_{k,n}$ is coprime to $g_{i}$ over $\Z$, for every $i$. The polynomials $h_{k,n}$ and $g_{i}$'s are all monic as polynomials in $b$ over $\Z[a]$. So for every $i$, the polynomials $g_{i}$ and $h_{k,n}$ are coprime over $\Q$. Hence, $h_{k,n}$ is coprime to $g_{i}$ over $\C$ too, for every $i$. As there are only finitely many $g_{i}$, the complement of $\mathscr{S}_{k,n}$ in the curve of $h_{k,n}$ is finite. Hence, $\mathscr{S}_{k,n}$ is Zariski dense subset of the curve of $h_{k,n}$ in $\mathcal{M}_{3}$, with finitely many points in its complement.
    \end{proof}
    From Lemma \ref{dense}, one directly obtains the following corollary,
    \begin{corollary}\label{irrcoro}
        For any $k\geq 0, n>0$, the set $\mathscr{S}_{k,n}$ is an algebraic curve. Also, if the polynomial $h_{k,n}$ is irreducible over $\C$, then the curve $\mathscr{S}_{k,n}$ is irreducible. \qed
    \end{corollary}
    \begin{corollary}\label{redcoro}
        Let $k \geq 0, n>0$. Let us assume that the polynomial $h_{k,n}$ factorizes in $\C[a,b]$ as $h_{k,n}= t\cdot t^{\diamond}$, where $t\in \C[a,b]$ is an irreducible polynomial and $t^{\diamond}(a,b):= t(-a,-b)$. Then, $\mathscr{S}_{k,n}$ is an irreducible curve in $\mathcal{M}_{3}$. 
    \end{corollary}
    \begin{proof}
        By definition of $t^{\diamond}$, the curves of $t, t^{\diamond}$ and $h_{k,n}$ in $\mathcal{M}_{3}$, coincide. Replacing $h_{k,n}$ with $t$ (or, $t^{\diamond}$) in Lemma \ref{dense} and Corollary \ref{irrcoro}, we obtain the corollary.
    \end{proof}
    \begin{remark}
        As we will see in section \ref{t''}, converse of the second part of Corollary \ref{irrcoro} is not true. For example, we will see that $h_{k,2}$ polynomials can be reducible over $\C$. It turns out that for any $k\geq 0$, $h_{k,2}$ can have at most two irreducible factors. We will further show that if $h_{k,2}$ is reducible for some $k\geq 0$, it must factorize as in the assumption of Corollary \ref{redcoro}. Thus, using Corollaries \ref{irrcoro} and \ref{redcoro}, we will show that $\mathscr{S}_{k,2}$ is an irreducible curve in $\mathcal{M}_{3}$, for any $k\geq 0$.  
    \end{remark}
\section{Notations}\label{notations}
    We will use the following notations for the rest of the article. Let $g,h$ be elements of $\Z[a,b],$ the polynomial ring in variables $a,b$ over $\Z$. 
    \begin{itemize}
        \item By saying $g$ is \emph{monic} in $\Z[a][b]$, we mean $g$ is monic as a polynomial in $b$ over the ring $\Z[a]$.
        \item By $\mbox{Res}(g,h)$, we denote the \emph{resultant} of $g$ and $h$, both considered as polynomials in $b$ with coefficients coming from $\Z[a]$. So, $\mbox{Res}(g,h)\in \Z[a]$.
    \end{itemize}
    Consider the polynomial $f$ as defined in Equation \eqref{newmaineq}. For any non-negative integers $k,n$, with $n>0$,  
    \begin{itemize}
        \item $f^{0}:=$ identity map, $f^{n}:=f^{n-1} \circ f, $ for all $ n\in \N$.
        \item $f'$ denote the derivative of $f$ with respect to $z$.
        \item $f_{k,n}=f_{k,n}(a,b):= f^{k+n}(a)-f^{k}(a)$.
        \item $h_{k,n}=h_{k,n}(a,b):= f_{k,n}/\prod_{i} g_{i}^{\alpha_{i}}$, where $g_{i}$ varies over all distinct irreducible factors of $f_{l,m}$ over $\Z$, where $l\leq k,\; m|n,\; (l,m)\neq (k,n)\in \Z^{2}$, and for each $i$, $\alpha_{i}$ is the highest power of $g_{i}$ that divides $f_{k,n}$.
        \item $\C^{2}:=$ complex affine space of dimension $2$.
        \item $\mathcal{M}_{3}:=\C^{2}\bigg/\biggl((a,b)\sim (-a,-b) \biggr)$.
        \item $\mathscr{S}_{k,n}:=$ the set of all points of $\mathcal{M}_{3}$ for which $a$ is strictly $(k,n)$-preperiodic.
        \item $G_{\Q}$ denotes the absolute Galois group of $\Q$.
        \item $\F_{3}$ denotes the finite field of order $3$.
    \end{itemize}
    Let $F$ be a number field and $g\in F[a,b]$.
    \begin{itemize}
        \item By saying $g$ has a smooth $F$-rational point, we mean that there exists a point $(a^{0},b^{0})\in F^{2}$, such that $g(a^{0},b^{0})=0$ and $g$ is smooth at $(a^{0},b^{0})$.    
    \end{itemize}
\section{Basic lemmas and Tools}\label{lemmas}
In this section, we gather a collection of lemmas and tools, that will be used in later sections. Generalizations of some statements of this section have been proved in \cite{Pat23}. For such statements, we omit the proof here and refer to the generalized statement in \cite{Pat23}.

Consider the Equation in \cite[Equation~$(2.2)$]{Pat23}. Putting degree $d=3$ and $\alpha_{1}= a, \alpha_{2}= -a, \beta= b$ in \cite[Equation~$(2.2)$]{Pat23}, we get the modified \emph{Branner Hubbard normal form} for monic reduced cubic polynomials,
    \begin{equation}
        f(z) = z^{3} -3a^{2}z + 2a^{3} +b.
    \end{equation}    
\subsection{Divisibility properties of $f_{k,n}$} 
    Let $k,l\in \Z$ such that $0\leq l\leq k$. 
    \begin{lemma}\label{l}
        Let $m,n \in \N$ such that $m$ divides $n$. The polynomial $f_{l,m}$ divides $f_{k,n}$ in $\Z[a,b]$.
    \end{lemma}
    \begin{proof}
        In \cite[Lemma~4.1]{Pat23}, replacing $\hat{f}_{k,n,d}, \hat{f}_{l,m,d}, \Z_{(p)}[\alpha_{1},\alpha_{2},...,\alpha_{d-2}, \beta]$ with $f_{k,n}, f_{l,m}, \Z[a,b]$ respectively, one obtains this lemma.
    \end{proof}
    \begin{lemma}\label{k}
        Let $m,n\in \N$, and g.c.d.$(m,n)=r$. Let $g$ be an irreducible element of $\Z[a,b]$, monic as a polynomial in  $\Z[a][b]$. If $g$ divides both $f_{k,n}$ and $f_{l,m}$ in $\Z[a,b]$, then $g$ divides $f_{l,r}$ in $\Z[a,b]$.
    \end{lemma}
    \begin{proof}
        Similarly as lemma \ref{l}, one obtains this lemma from \cite[Lemma~$4.2$]{Pat23}.
    \end{proof}
    From Lemmas \ref{l} and \ref{k}, one directly obtains the following corollary,
    \begin{corollary}
        Let $m,n\in \N$, and g.c.d.$(m,n)=r$. The polynomial $f_{l,r}$ divides $g.c.d.(f_{k,n},f_{l,m})$ in $\Z[a,b]$. Moreover, The radical ideals of the ideal generated by $f_{l,r}$ and the ideal generated by $g.c.d.(f_{k,n},f_{l,m})$ are the same. \qed
    \end{corollary}
\subsection{A weak version of Thurston's rigidity theorem for $\mathcal{M}_{3}$}
    \begin{theorem}\label{thurs}
        Fix $k_{1},k_{2}\in \N \cup \{0\}$, and $n_{1},n_{2}\in \N$. Then, the polynomials 
        \begin{equation*}
            f^{k_{1}+n_{1}}(a) - f^{k_{1}}(a) \text{ and } f^{k_{2}+n_{2}}(-a) - f^{k_{2}}(-a)
        \end{equation*}
        are coprime in $\C[a,b]$.
    \end{theorem}
    \begin{proof}
        In the version of Thurston's rigidity theorem stated in \cite[Theorem~$4.4$]{Pat23}, replacing $$\hat{f}^{k_{1}+n_{1}}(\alpha_{1}) - \hat{f}^{k_{1}}(\alpha_{1})\;\text{ and }\; \hat{f}^{k_{i}+n_{i}}(\alpha_{i}) - \hat{f}^{k_{i}}(\alpha_{i})$$  with $$f^{k_{1}+n_{1}}(a) - f^{k_{1}}(a)\;\text{ and }\; f^{k_{2}+n_{2}}(-a) - f^{k_{2}}(-a)$$ respectively, one obtains this theorem.
    \end{proof}
\subsection{Generalised Eisenstein Irreducibility criterion}
    \begin{theorem}\label{m}
        Let $g,h$ be non-constant elements of $\Z[a,b]$, both monic as elements of $\Z[a][b]$. Let $\mbox{Res}(g,h)$ denote the resultant of $g$ and $h$, both considered as polynomials in $b$ over the integral domain $\Z[a]$. Suppose the following conditions hold,\\
        1) $g\equiv h^{n}\;(\text{mod } 3)$, for some $n\in \N$.\\
        2) $h\;(\text{mod } 3)$ is irreducible in $\F_{3}[a,b]$.\\
        3) $\mbox{Res}(g,h)\not\equiv 0\;(\text{mod }3^{2\cdot deg(h)})$, where $deg(h)$ is the degree of $h$ as a polynomial in $b$ over $\Z[a]$.
    
        Then, $g$ is irreducible in $\Q[a,b]$.
    \end{theorem}
    \begin{proof}
        Replacing $p$ and $\Z[\alpha_{1},...,\alpha_{p^{e}-2}, \beta]$ with $3$ and $\Z[a,b]$ respectively in \cite[Theorem~$4.6$]{Pat23}, this theorem follows.
    \end{proof}
\subsection{Extension of irreducibility}
    In this subsection, we relate the irreducibility of a multivariate polynomial over a number field and over $\C$. The results in this subsection are borrowed from the article \cite{article1} by Buff, Epstein and Koch. We should mention that while we state Theorem \ref{n} and Corollaries \ref{o}, \ref{p} for polynomials in two variables, they can be directly generalised for polynomials in any number of variables.
    \begin{theorem}{\cite[Lemma~$5$]{article1}}\label{n}
        Let $g$ be an element of $\Q[a,b]$. Let $g(0,0)=0$, and the linear part of $g$ is non-zero. Then,
    
            $g$ is irreducible in $\Q[a,b] \iff g$ is irreducible in $\C[a,b]$. \qed
    \end{theorem}
    \begin{corollary}\label{o}
        Let $g$ be an element of $\Q[a,b]$. Let us assume that $g$ has a smooth $\Q$-rational point, i.e. there exists a point $(a^{0},b^{0})\in \Q^{2}$, such that $g(a^{0},b^{0})=0$, and $g$ is smooth at $(a^{0},b^{0})$. Then,
    
            $g$ is irreducible in $\Q[a,b] \iff g$ is irreducible in $\C[a,b]$.
    \end{corollary}
    \begin{proof}
        By an affine change of coordinate, sending $(a^{0},b^{0})$ to $(0,0)$, from $g$ one obtains a polynomial $g'\in \Q[a,b]$, such that constant term of $g'$ is zero and $g'$ has non-zero linear part. Also, $g$ is irreducible over $\C $ (or, over $\Q) \iff g'$ is irreducible over $\C$ (or, over $\Q$). Applying Theorem \ref{n} on $g'$, one obtains the corollary.
    \end{proof}
    \begin{corollary}\label{p}
        Let $F$ be a number field, which means finite extension over $\Q$. Let $g$ be an element of $F[a,b]$. Let us assume that $g$ has a smooth $F$-rational point, defined similarly as in the previous corollary. Then,

            $g$ is irreducible in $F[a,b] \iff g$ is irreducible in $\C[a,b]$.
    \end{corollary}
    \begin{proof}
        Replacing $\Q$ with $F$ in the proofs of Theorem \ref{n} and Corollary \ref{o}, every argument there follows verbatim, and one obtains this corollary. 
    \end{proof}
\subsection{Even and odd polynomials}\label{cubic}
    \begin{definition}
        Let $g$ be an element in $\C[a,b]$. we say $g$ is \emph{even polynomial} iff $g(a,b)=g(-a,-b)$, and $g$ is \emph{odd polynomial} iff $g(a,b)=-g(-a,-b)$.
    \end{definition}
    Every non-zero polynomial $g\in \C[a,b]$, can be written as $g=g_{e}+g_{o}$, where $g_{e}\in \C[a,b]$ is an even polynomial, and $g_{o}\in \C[a,b]$ is an odd polynomial.

    Let $G_{e}(\text{and},\; G_{o})$ denote the set of all even (and, odd) polynomials in $\C[a,b]$.
    \begin{lemma}\label{evenodddecomp}
        The sets $G_{e}, G_{o}$ are additive subgroups of $\C[a,b]$. The set $G:= G_{e}\cup G_{o}$ is closed under multiplication.  Also, if $g_{1},g_{2}\in G$, and $g_{1}=g_{2}\cdot h$, for some $h\in \C[a,b]$, then $h$ belongs to $G$.
    \end{lemma}
    \begin{proof}
        Only the last part of the lemma is non-trivial. We will prove the last part by contradiction. Let us assume that $h$ is neither even nor odd polynomial. So, $h$ admits an even-odd decomposition $h= h_{e} + h_{o}$, where $h_{e}$ is even polynomial, $h_{o}$ is odd polynomial and $h_{e} \neq 0 \neq h_{o}$. Now, $g_{2}$ being even or odd polynomial, $g_{1}$ admits an even-odd decomposition $g_{1} = g_{2} \cdot h_{e} + g_{2} \cdot h_{o}$, where $g_{2} \cdot h_{e} \neq 0 \neq g_{2}\cdot h_{o}$. Hence, we arrive at a contradiction. 
    \end{proof}
    \begin{lemma} \label{r}
        For $k\in \N\cup \{0\},n\in \N$, the polynomials $f_{k,n}$ are odd polynomials.
    \end{lemma}
    \begin{proof}
        Let $l\geq 0$. Consider the polynomial $f^{l}(z)\in \Z[z,a,b]$. Every monomial term of $f^{l}(z)$ is of odd degree. Hence, same is true for $f^{l}(a)$. Therefore, $f^{l}(a)$ is an odd polynomial in $\Z[a,b]$ for any $l\geq 0$ and so is $f_{k,n}= f^{k+n}(a)-f^{k}(a)$, for any $k \geq 0,n \in \N$. 
    \end{proof}
    \begin{corollary}\label{s}
        Let $k\in \N\cup \{0\}, n\in \N$. If the polynomials $h_{l,m}$ are irreducible over $\Q$ for all $0\leq l\leq k,\; 1\leq m \leq n,\; (l,m)\neq (k,n) \in \Z^{2}$, then the polynomial $h_{k,n}$ is even or odd polynomial.
    \end{corollary}
    \begin{proof}
        If the polynomials $h_{l,m}$ are irreducible over $\Q$ for all $0\leq l\leq k,\; 1\leq m,\; m|n,\; (l,m)\neq (k,n) \in \Z^{2}$, then for any such $(l,m)$ (including $(k,n)$), one can write
        \begin{equation*}
            f_{l,m}= h_{l,m} \cdot \prod_{\substack{0\leq j\leq l,\; 1\leq r,\; r|m,\\ (j,r)\neq (l,m) \in \Z^{2}}} h_{j,r}^{a_{j,r}}, \text{ for some } a_{j,r}\in \N
        \end{equation*}
        Observe that $h_{0,1}=f_{0,1}=b-a$ is an odd polynomial. Applying induction on both $l$ and $m$ such that $0\leq l\leq k, 1\leq m, m|n$, and using the last part of Lemma \ref{evenodddecomp}, one gets that $h_{k,n}$ is even or odd polynomial.
    \end{proof}
    \begin{lemma}\label{t}
        Let $g\in \C[a,b]$ be an even or odd polynomial. Let $h\in \C[a,b]$ be an irreducible polynomial, having a decomposition $h=h_{e}+h_{o},$ where $h_{e}$ is even polynomial, and $h_{o}$ is odd polynomial. Let $h^{\diamond}:= h_{e}-h_{o}= h(-a,-b)$. Then the following statements are true,

            1) $h^{\diamond}$ is irreducible.

            2) In $\C[a,b]$, $h$ divides $g \iff h^{\diamond}$ divides $g$.

            3) If $h$ is not even or odd polynomial, then $h$ and $h^{\diamond}$ are distinct (which means not equal upto associates) irreducible polynomials in $\C[a,b]$.
    \end{lemma}
    \begin{proof}
        First two parts of the lemma follows from using the change of variable $(a,b) \rightarrow (-a,-b)$.

        For the third part of the lemma, if $h$ is not even or odd polynomial, then $h_{e}\neq 0 \neq h_{o}$. So, if $h$ and $h^{\diamond}$ are constant multiple of each other, then $h$ divides $h\pm h^{\diamond}$, which are the polynomials $2h_{e},2h_{o}$. As either $deg(h_{e})<deg(h)$ or $deg(h_{o})<deg (h)$, we get a contradiction.
    \end{proof}

\section{Irreducibility of $\mathscr{S}_{k,2}$ curves}\label{t''}
    From Equation \eqref{newmaineq}, we have the normal form
    \begin{equation}\label{maineqcopy}
        f(z) = z^{3} - 3a^{2}z + 2a^{3} + b, 
    \end{equation}
    with $\pm a$ as finite critical points.

    One observes the following factorization,
    \begin{equation}\label{u'}
        f(z)-f(w)= (z-w)(z^{2}+zw+w^{2}-3a^{2}).
    \end{equation}
    We first study the polynomial $h_{1,2}$. This will give us a glimpse into the general nature of the polynomials $h_{k,2},\; k \geq 0$. 
    \begin{lemma}\label{h123}
        The polynomial $h_{1,2}$ is $(b-a)^{2}+1$, which has the following properties:
        \begin{itemize}
            \item It is irreducible over $\Q$,
            \item It is reducible and smooth over $\C$,
            \item There is no $\Q$-rational point on the solution curve of $h_{1,2}$ in $\C^{2}$,
            \item The curve, $\mathscr{S}_{1,2}$, of $h_{1,2}$ in $\mathcal{M}_{3}$ is an irreducible line.
        \end{itemize}
    \end{lemma}
    \begin{proof}
        To obtain $h_{1,2}$, we need to factor out all irreducible factors of $f_{0,2}$ and $f_{1,1}$ from $f_{1,2}$ with each irreducible factor raised to their highest power that divides $f_{1,2}$. We compute,  
        \begin{align}
            f_{0,2} & =f(b)-a \notag\\
            & = b^{3} -3a^{2}b+2a^{3}+b-a \notag\\
            & = (b-a)\left((b-a)(b+2a)+1\right). \label{h023.}\\
            f_{1,1} & =f(b)-b \notag \\
            & =b^{3}-3a^{2}b+2a^{3}\notag \\
            & =(b-a)^{2}(b+2a).\label{f113}\\
            f_{1,2} & =f^{3}(a)-f(a)\notag \\
            & =(f^{2}(a)-a)\left((f^{2}(a))^{2}+af^{2}(a)-2a^{2}\right)\notag \\
            & = f_{0,2} \left((f(b))^{2}+af(b)-2a^{2}\right)\notag \\
            & = f_{0,2}\; (f(b)-a)(f(b)+2a)\notag \\
            & = (f_{0,2})^{2} \;(f(b)+2a)\label{f123byf023}\\
            & = (f_{0,2})^{2}\left(b^{3}-3a^{2}b+2a^{3}+b+2a \right)\notag \\
            & = (f_{0,2})^{2} (b+2a)\left((b-a)^{2}+1\right).\label{f123}
        \end{align}
        
        As  $(b-a)^{2}+1$ is irreducible over $\Q$, and $(b+2a)$ is a factor of $f_{1,1}$, we get 
        \begin{align}
            h_{1,2} & =(b-a)^{2}+1 \label{h12}\\
            & = (b-a+i)(b-a-i) \notag
        \end{align}
        Let us define $l_{1}(a,b):=(b-a+i),\; l_{2}(a,b):=(b-a-i)$. Now, it directly follows that $h_{1,2}$ is irreducible over $\Q$, reducible and smooth over $\C$, and has no $\Q$-rational point on it. Also, $l_{1}(-a,-b)=-l_{2}(a,b)$. As $\mathcal{M}_{3}$ is obtained by identifying $(a,b)$ with $(-a,-b)$ on $\C^{2}$, the lines $l_{1}$ and $l_{2}$ merge together in $\mathcal{M}_{3}$, making $\mathscr{S}_{1,2}$ an irreducible line in $\mathcal{M}_{3}$.  
    \end{proof}
    Next, we will show that $h_{0,2}$ is irreducible over $\C$. 
    \begin{lemma}\label{h023}
        The polynomial $h_{0,2}$ is $(b-a)(b+2a)+1$, and it is irreducible over $\C$.
    \end{lemma}
    \begin{proof}
        From Equation \eqref{h023.} and the fact that $f_{0,1}=b-a$, we get that $h_{0,2} = (b-a)(b+2a)+1$. By a change of variable, one sees that $(b-a)(b+2a) +1$ is irreducible in $\C[a,b]$ iff $xy+1$ is irreducible in $\C[x,y]$. Hence, the lemma is proved.  
    \end{proof}
    In the last two lemmas, we have seen $h_{0,2}$ and $h_{1,2}$ are irreducible over $\Q$. Next, we study the irreducibility of $h_{k,2}$ polynomials, over $\Q$, where $k$ varies over all natural numbers greater than $1$. We will show that $h_{k,2}$ is $3$-Eisenstein with respect to the polynomial $h_{1,2}$. For that, we need to check that the three conditions in generalised Eisenstein irreducibility criterion (Theorem \ref{m}) hold. First, we check that the condition $1$ of Theorem \ref{m} holds.
    \begin{lemma}\label{k23E1}
        For any $k\in \N$, $h_{k,2}\equiv (h_{1,2})^{N_{k}}\;(\text{mod }3),$ for some $N_{k}\in \N$.
    \end{lemma}
    \begin{proof}
        From Equation \eqref{maineqcopy}, we have $f \equiv z^{3} - a^{3} +b\;(\text{mod }3)$. Hence, 
        \begin{align*}
             f_{k,2} & = f^{k+2} (a) - f^{k}(a)\\
             & = f^{k+1}(b) - f^{k-1}(b)\\
             & \equiv \sum_{i=k}^{k+1}(b-a)^{3^{i}}\; (\text{mod }3)\\
             & \equiv (b-a)^{3^{k}}\left((b-a)^{2\cdot 3^{k}}+1\right)\; (\text{mod }3)\\
             & \equiv (b-a)^{3^{k}}\left((b-a)^{2}+1\right)^{3^{k}}\;(\text{mod }3).
        \end{align*}
        Similarly, $f_{k,1}\equiv (b-a)^{3^{k}}\;(\text{mod }3)$. As $h_{k,2}$ divides $$\frac{f_{k,2}}{f_{k,1}}\equiv  \left((b-a)^{2}+1 \right)^{3^{k}} (\text{mod }3)$$ and the polynomial $(b-a)^{2}+1$ is irreducible modulo $3$, we have $h_{k,2}\equiv ((b-a)^{2}+1)^{N_{k}}\equiv (h_{1,2})^{N_{k}}\;(\text{mod }3)$ (by Lemma \ref{h123}), for some $N_{k}\in \N$. 
    \end{proof}
    As $h_{1,2}$ is irreducible modulo $3$, condition $2$ of generalised Eisenstein irreducibility criterion (Theorem \ref{m}) holds for $h_{1,2}$. For condition $3$ of Theorem \ref{m}, we need to study the resultant $\mbox{Res}(h_{k,2},h_{1,2})$, where $k\in \N, k>1$. To do that, we require some divisibility properties of $f_{k,2}$ and $f_{0,k}$, which we study in the Lemma \ref{divlemma}. 

    Let $g_{1},g_{2}\in \C[a,b]$. Let $o(g_{1},g_{2}):=\alpha,$ such that $\alpha\in \N\cup \{0\}, (g_{2})^{\alpha}|g_{1}, (g_{2})^{\alpha + 1}\nmid g_{1}$. In other words, $o(g_{1},g_{2})$ is defined to be the highest power of $g_{2}$ that divides $g_{1}$ in $\C[a,b]$.
    \begin{lemma}\label{divlemma}
        For any $k\in \N$, we have $o(f_{k,2},f_{0,2})\geq 2$. For any even $k\in \N$, we have $o(f_{0,k},f_{0,2})=o(f_{0,k},h_{0,2})=1$.
    \end{lemma}
    \begin{proof}
        From Equation \eqref{f123byf023}, we get $f_{1,2}= (f_{0,2})^{2}\cdot \left(f(b)+2a\right)$. So, $o(f_{1,2},f_{0,2})\geq 2$. As $f_{1,2}$ divides $f_{k,2}$ for any $k\in \N$, the first part of the lemma follows.

        For even $k\in \N$,  we know that $o(f_{0,k},h_{0,2})\geq o(f_{0,k},f_{0,2})\geq 1$. Hence to prove both equalities of the lemma, it is enough to show that $o(f_{0,k},h_{0,2})=1$.
        As $k$ is even, let $k=2l,l\in \N$. Observe that,
        \begin{align*}
            f_{0,2l} & =f^{2l}(a)-a\\
            & =\sum_{i=1}^{l} \left(f^{2i}(a)-f^{2i-2}(a)\right)\\
            & =\sum_{i=1}^{l} f_{2i-2,2}.
        \end{align*}
        From the first part of this lemma, $o(f_{2i-2,2},h_{0,2})\geq o(f_{2i-2,2},f_{0,2})\geq 2$, for $i\in \N, i > 1$. Also, $o(f_{0,2},h_{0,2})=1$. Hence, $o(f_{0,k},h_{0,2})=o(f_{0,2l},h_{0,2})=1$.
    \end{proof}
    In the next lemma and the following corollary, we establish condition $3$ of generalised Eisenstein irreducibility criterion (Theorem \ref{m}) for $h_{k,2}$ and $h_{1,2}$.
    \begin{lemma}\label{k23res}
        Let $l=b-a+i \in \C[a,b]$. Up to multiplication by a power of $i$, The resultant $\mbox{Res}(h_{k,2},l)$ is,
        \begin{center}
            $\mbox{Res}(h_{k,2},l)= \left\{ \begin{array}{rcl}
                3(2ai+1); & k \mbox{ even,} & k>0, \\ 3a; & k \mbox{ odd,} & k>1. \\
            \end{array}\right.$
        \end{center}
    \end{lemma}
    \begin{proof}
        Let $k\in \N, k > 1$. We will first remove irreducible factors of $f_{k-1,2}$ in $\Z[a,b]$, from $f_{k,2}$ with each such factor raised to the highest power that divides $f_{k,2}$.
        Consider the polynomial, 
        \begin{align}\label{hk231}
             g_{k}(a,b):= \frac{f_{k,2}}{f_{k-1,2}} & = \frac{f^{k+1}(b)-f^{k-1}(b)}{f^{k}(b)-f^{k-2}(b)}\notag \\
             & = \left(f^{k}(b)\right)^{2}+f^{k}(b)f^{k-2}(b)+\left(f^{k-2}(b)\right)^{2}- 3a^{2}\notag \\
             & \equiv 3\left(\left(f^{k-2}(b)\right)^{2}-a^{2}\right)\; \left( \text{mod } f_{k-1,2}\right).
        \end{align}
        So, any irreducible polynomial $s_{k}\in \Z[a,b]$ that divides both $g_{k}$ and $f_{k-1,2}$, will also divide 
        \begin{equation}\label{pereq}
            3\left((f^{k-2}(b))^{2}-a^{2}\right)=3 f_{0,k-1} (f^{k-2}(b)+a).
        \end{equation}
        From Thurston's rigidity theorem (Theorem \ref{thurs}), we get that $f^{k-2}(b)+a$  and $f_{k-1,2}$ are coprime. So, $s_{k}$ will divide $f_{0,k-1}$ (we can remove $3$, because $f_{k-1,2}$ is monic in $\Z[a][b]$, and so are its irreducible factors). As $s_{k}$ divides both $f_{0,k-1}$ and $f_{k-1,2}$, by Lemma \ref{k} we have that $s_{k}$ divides $f_{0,1}$ if $k$ is even, and $f_{0,2}$ if $k$ is odd.

        \textbf{Let $k\in \N$ be even.} As $f_{0,1}= b-a$, we get that 
        $$h_{k,2}\; \text{ divides }\; \frac{g_{k}(a,b)}{(b-a)^{i_{k}}},$$ where $i_{k}$ is the highest power of $(b-a)$ that divides $g_{k}$. Also, 
        $$\frac{g_{k}(a,b)}{(b-a)^{i_{k}}}\;\text{ is coprime to }\;f_{k-1,2}.$$

        \textbf{Let $k\in \N, k>1, k$ odd.} 
        From Equation \eqref{h023.} and Lemma \ref{h023}, we know that, over $\Q$, the irreducible factors of $f_{0,2}$ are $h_{0,2}$ and $(b-a)$. From Equations \eqref{hk231}, \eqref{pereq} and Lemma \ref{divlemma}, we have $o(g_{k},h_{0,2})=1$, for all $k\in \N, k>1$. So, 
        $$h_{k,2}\; \text{ divides }\; \frac{g_{k}(a,b)}{h_{0,2}\cdot (b-a)^{i_{k}}},$$ 
        where $i_{k}$ is the highest power of $(b-a)$ that divides $g_{k}$. Also,
        $$\frac{g_{k}(a,b)}{h_{0,2}\cdot (b-a)^{i_{k}}}\;\text{ is coprime to }\;f_{k-1,2}.$$

        For $k\in \N, k>1$, let us define \begin{equation}\label{hk232}
            g'_{k}(a,b) := \left\{ \begin{array}{rcl}
                g_{k}(a,b)/(b-a)^{i_{k}}; & k &\mbox{even} \\ g_{k}(a,b)/(h_{0,2}\cdot (b-a)^{i_{k}}); & k & \mbox{odd} \\
            \end{array}\right.
        \end{equation}

        Next, we will factor out the irreducible factors of $f_{k,1}$ from $g'_{k}(a,b)$. But irreducible factors of $f_{k-1,1}$ has already been factored out, since $f_{k-1,1}$ divides $f_{k-1,2}$. Hence, we need to consider the irreducible factors of 
        \begin{equation}\label{5.10}
            \frac{f_{k,1}}{f_{k-1,1}}=(f^{k}(a))^{2}+f^{k}(a)f^{k-1}(a)+(f^{k-1}(a))^{2}-3a^{2},
        \end{equation}
        and their highest powers that divide $g'_{k}(a,b)$. Let's denote the product of common irreducible factors of $f_{k,1}/f_{k-1,1}$ and $g'_{k}$, each raised to their highest power that divides $g'_{k}$, as $t_{k}(a,b)$. By definition of $h_{k,2}$,
        \begin{equation}\label{hk233}
            h_{k,2}(a,b) = \frac{g'_{k}(a,b)}{t_{k}(a,b)}.
        \end{equation}

        Now, we can compute the resultant. We have, $\mbox{Res}(h_{k,2},l)=h_{k,2}(a,a-i)$.

        Putting $b=a-i$ in Equation \eqref{maineqcopy}, by direct computation or by using Lemma \ref{h123} one obtains $f^{k}(a)=a-i, $ for odd $k\in \N$, and $f^{k}(a)=-2a, $ for even $k\in \N$. Moreover, $h_{0,2}(a,a-i)= -i(-i+3a)+1=-3ai$.

        Observe that for any $k\in \N$,
        \begin{align*}
            \left(\frac{f_{k,1}}{f_{k-1,1}}\right)(a,a-i) & = (a-i)^{2}-2a(a-i)+4a^{2}-3a^{2}\\
            & =-1.
        \end{align*}
        So, $t_{k}(a,a-i)=1$, up to multiplication by a power of $i$ (because, any irreducible factor $t'_{k}$ of $t_{k}$ in $\Z[a,b]$ divides $f_{k,1}/f_{k-1,1}$ in $\Z[a,b]$. So in $\Z[i][a]$, $t'_{k}(a,a-i)$ divides $(f_{k,1}/f_{k-1,1})(a,a-i)=-1$).

        From the above discussion, we get that, up to multiplication by a power of $i$,

        \textbf{For $k$ even,}
        \begin{align*}
            \mbox{Res}(h_{k,2},l) & =h_{k,2}(a,a-i) & \\
            & = g_{k}(a,a-i) & (\mbox{by Equation }\eqref{hk232})\\
            & = 3(a-i)^{2}-3a^{2} & (\mbox{by Equation }\eqref{hk231})\\
            & =-3(2ai+1). & 
        \end{align*}
        
        \textbf{For $k> 1, k$ odd,}
        \begin{align*}
            \mbox{Res}(h_{k,2},l) & = h_{k,2}(a,a-i) & \\
            & = \frac{g_{k}(a,a-i)}{h_{0,2}(a,a-i)} & (\mbox{by Equation }\eqref{hk232})\\
            & =\frac{3(4a^{2})-3a^{2}}{-3ai} & (\mbox{by Equation }\eqref{hk231})\\
            & = 3ai. & 
        \end{align*}
        Hence, the lemma is proved.
    \end{proof}
    \begin{remark}
        Showing that $\mbox{Res}(h_{k,2},l)$ divides $3(2ai+1)$ for $k$ even, and $3a$ for $k$ odd, $k>1$, is all one needs to prove the irreducibility of $h_{k,2}$ in $\Q[a,b]$. A proof of that statement would be much shorter. But as we will see later, proving the equality, more precisely,  proving that $\mbox{Res}(h_{k,2},l)$ is not constant allows us to prove irreducibility of $\mathscr{S}_{k,2}$ in $\mathcal{M}_{3}$.
    \end{remark}
    \begin{corollary}\label{k23res2}
        The resultant $\mbox{Res}(h_{k,2},h_{1,2})\not\equiv 0\;($mod $81)$, for any $k\in \N, k>1$.
    \end{corollary}
    \begin{proof}
        By Lemma \ref{h123}, we have
        \begin{align*}
            \mbox{Res}(h_{k,2},h_{1,2}) & =\mbox{Res}(h_{k,2},b-a+i)\cdot \mbox{Res}(h_{k,2},b-a-i)\\
            & =h_{k,2}(a,a-i)\cdot h_{k,2}(a,a+i).
        \end{align*}
        Complex conjugation of $h_{k,2}(a,a+i)$ is $h_{k,2}(a,a-i)$. Hence, from Lemma \ref{k23res}, upto multiplication by $\pm 1$, we have for $k$ even, 
        \begin{align*}
            \mbox{Res}(h_{k,2},h_{1,2}) & = 3(2ai+1)\cdot 3(-2ai+1)\\
            & = 9(4a^{2}+1),
        \end{align*}
        and for $k$ odd, $k>1$, 
        $$\mbox{Res}(h_{k,2},h_{1,2})=9a^{2}.$$ None of them is divisible by $81$. Hence, the corollary is proved.
    \end{proof}
    Next, we put together all the previous lemmas and corollary of this section, along with generalised Eisenstein irreducibility criterion (Theorem \ref{m}), to show that $h_{k,2}$ is irreducible over $\Q$, for every choice of non-negative integer $k$.
    \begin{theorem}\label{irrqk23}
        For each $k\in \Z, k \geq 0$, the polynomial $h_{k,2}$ is irreducible in $\Q[a,b]$. 
    \end{theorem}
    \begin{proof}
        From Lemmas \ref{h123} and \ref{h023}, we know that $h_{0,2},h_{1,2}$ are irreducible over $\Q$. Let $k\in \N, k>1$. Putting $g=h_{k,2}, h=h_{1,2}$ in generalised Eisenstein irreducibility criterion (Theorem \ref{m}), from Lemmas \ref{h123}, \ref{k23E1} and Corollary \ref{k23res2}, we get that $h_{k,2}$ is irreducible in $\Q[a,b]$.  
    \end{proof}
    Next, we use irreducibility of $h_{k,2}$ and $h_{k,1}$ over $\Q$, to show that $h_{k,2}$ is even for every $k\geq 0$. We will need this following corollary in the proof of Theorem \ref{irrck23}. 
    \begin{corollary}\label{k23even}
        For each $k\in \Z, k\geq 0$, the polynomial $h_{k,2}$ is even polynomial.
    \end{corollary}
    \begin{proof}
        For $k=0,1$, the corollary follows from Lemmas \ref{h023}, \ref{h123}, respectively. Let $k\in \N, k>1$. From \cite[Theorem~$5.8$]{Pat23}, we get that $h_{l,1}$ polynomials are irreducible polynomials over $\Q$ for every choice of non-negative integer $l$. From Theorem \ref{irrqk23} above, for any $l\geq 0$, the polynomial $h_{l,2}$ is irreducible over $\Q$.  Hence, using Corollary \ref{s} we get that $h_{k,2}$ is an even or odd polynomial. From Lemmas \ref{h123}, \ref{k23E1},  we get $h_{k,2}(0,0)\equiv h_{1,2}(0,0)^{N_{k}} \equiv 1\;(\text{mod }3)$. So, the polynomial $h_{k,2}$ has a non-zero constant term. Hence, $h_{k,2}$ is an even polynomial, for any $k\geq 0$.
    \end{proof}
    Now, we will show that although $h_{k,2}$ might not be irreducible over $\C$, all the curves $\mathscr{S}_{k,2}$ are irreducible in $\mathcal{M}_{3}$.
    \begin{theorem}\label{irrck23}
        For each $k\in \Z, k \geq 0$, the curve $\mathscr{S}_{k,2}$ is irreducible.
    \end{theorem}
    \vspace{-0.1in}
    \begin{proof}
        From Lemmas \ref{h123}, \ref{h023}, we know that $\mathscr{S}_{0,2}, \mathscr{S}_{1,2} \subset \mathcal{M}_{3}$ are irreducible.

        Let $k>1$. From Lemma \ref{k23res}, we have that $h_{k,2}$ intersects the line $l= b-a+i$ at $(i/2,-i/2)$ point, for $k$ even, and at $(0,-i)$ point, for $k$ odd. As $\mbox{Res}(h_{k,2},b-a+i)$ is linear polynomial in $a$ (Lemma \ref{k23res}), $h_{k,2}$ is smooth at $(i/2,-i/2)$ point, for $k$ even, and at $(0,-i)$ point, for $k$ odd. So, for any $k>1$, the polynomial $h_{k,2}$ has a smooth $\Q[i]$-rational point.

        Let us assume that $h_{k,2}$ is irreducible in $\Q[i][a,b]$. As the polynomial $h_{k,2}$ has a smooth $\Q[i]$-rational point, by Corollary \ref{p}, we have that $h_{k,2}$ is irreducible in $\C[a,b]$. Hence, the curve $\mathscr{S}_{k,2}\in \mathcal{M}_{3}$ is irreducible.

        Next, let us assume that $h_{k,2}$ is reducible in $\Q[i][a,b]$. As $h_{k,2}$ is irreducible in $\Q[a,b]$, we get that $h_{k,2}= t_{k,2} \cdot \bar{t}_{k,2}$, for some irreducible polynomial $t_{k,2}\in \Q[i][a,b]$, and $\bar{t}_{k,2}$ is the complex conjugate of $t_{k,2}$.

        Now, for $k\in \N, k$ even, $h_{k,2}$ passes through the point $(i/2,-i/2)$. Without loss of generality, let us assume that $t_{k,2}(i/2,-i/2)=0$. By complex conjugation, we get $\bar{t}_{k,2}(-i/2,i/2)=0$. Also, $\bar{t}_{k,2}(i/2,-i/2)\neq 0$, otherwise $h_{k,2}$ will not be smooth at $(i/2,-i/2)$. Hence, $\bar{t}_{k,2}$ is not even or odd polynomial. But $h_{k,2}$ is even polynomial, by Corollary \ref{k23even}. As $\bar{t}_{k,2}$ is irreducible, from Lemma \ref{t}, we get that $t_{k,2}^{\diamond} = \bar{t}_{k,2}$, \ie  $t_{k,2}(-a,-b)=\bar{t}_{k,2}(a,b)$. So, the curves of $t_{k,2}$ and $\bar{t}_{k,2}$ in $\C^{2}$ merge together under the quotient map $\C^{2}\rightarrow \mathcal{M}_{3}$, and it is the same as the curve of $h_{k,2}$ in $\mathcal{M}_{3}$. Hence, if $t_{k,2}$ is irreducible over $\C$, then the curve of $h_{k,2}$ in $\mathcal{M}_{3}$, which is $\mathscr{S}_{k,2}$, is irreducible.

        So, to prove irreducibility of $\mathscr{S}_{k,2}$, it is enough to prove that $t_{k,2}$ is irreducible in $\C[a,b]$. Now, $t_{k,2}$ is an irreducible polynomial in $\Q[i][a,b]$, with a smooth $\Q[i]$-rational point, namely $(i/2,-i/2)$. Hence, by Corollary \ref{p}, we have $t_{k,2}$ is irreducible in $\C[a,b]$. So, for even $k \in \N$, $\mathscr{S}_{k,2}$ is irreducible in $\mathcal{M}_{3}.$

        Replacing the point $(i/2,-i/2)$ with $(0,-i)$ in the last two paragraphs, every argument there follows verbatim and we get that for odd $k>1$, $\mathscr{S}_{k,2}$ curves are irreducible in $\mathcal{M}_{3}$.
        Hence, the theorem is proved.
    \end{proof} 
\section{Explicit form of $h_{k,2}$ polynomials}\label{explicit}
In this section, we deduce the explicit form of $h_{k,2}$ polynomials, for any $k\geq 0$. We have already noted the explicit forms of $h_{0,2}$ and $h_{1,2}$ polynomials, in Lemmas \ref{h023} and \ref{h123}, respectively.
\begin{proposition}\label{exp}
    Let $k\in \N$. For odd $k$,
    \begin{equation}
        h_{k,2} = \frac{f_{k,2}\cdot f_{k-1,1}}{h_{0,2}\cdot f_{k-1,2}\cdot f_{k,1}},
    \end{equation}
    and for even $k$,
    \begin{equation}
        h_{k,2} = \frac{f_{k,2}\cdot f_{k-1,1}}{f_{k-1,2}\cdot f_{k,1}}.
    \end{equation}
\end{proposition}
\begin{proof}
    Let $k=1$. From Equations \eqref{h023.}, \eqref{f113}, \eqref{f123} and \eqref{h12}, we get 
    \begin{equation}
        h_{1,2} = \frac{f_{1,2}\cdot f_{0,1}}{h_{0,2}\cdot f_{0,2}\cdot f_{1,1}}.
    \end{equation}
    Let $k\in \N, k>1$. First, we will show that the polynomial $f_{k,2}/f_{k-1,2}$ is divisible by $f_{k,1}/f_{k-1,1}$. From Equations \eqref{hk231} and \eqref{5.10},
    \begin{align}
        \frac{f_{k,2}}{f_{k-1,2}} & = (f^{k}(b))^{2} + f^{k}(b)f^{k-2}(b) + (f^{k-2}(b))^{2} - 3a^{2} \label{6.3}\\
        \frac{f_{k,1}}{f_{k-1,1}} & = (f^{k-1}(b))^{2} + f^{k-1}(b)f^{k-2}(b) + (f^{k-2}(b))^{2} - 3a^{2}\label{6.4}
    \end{align}
    Subtracting Equation \eqref{6.4} from Equation \eqref{6.3}, we get
    \begin{equation}\label{6.5}
        \frac{f_{k,2}}{f_{k-1,2}} =  \frac{f_{k,1}}{f_{k-1,1}} + f_{k,1} \cdot (f^{k}(b) + f^{k-1}(b) + f^{k-2}(b)).
    \end{equation}
    From Equation \eqref{6.5}, we get that $f_{k,1}/f_{k-1,1}$ divides $f_{k,2}/f_{k-1,2}$ in $\Z[a,b]$. Furthermore, 
    \begin{equation}
        l_{k}:= \frac{f_{k,2}\cdot f_{k-1,1}}{f_{k-1,2}\cdot f_{k,1}} = 1 + f_{k-1,1} \cdot (f^{k}(b) + f^{k-1}(b) + f^{k-2}(b)).
    \end{equation}
    
    So, no factor of $f_{k-1,1}$ divides $l_{k}$. From \cite{Pat23}, we know that $h_{k,1}$ polynomials are irreducible over $\C$ with zero constant term. As $l_{k}$ has constant term $1$, $h_{k,1}$ does not divide $l_{k}$. So, no irreducible factor of $f_{k,1}$ divides $l_{k}$. As $(b-a)$ divides $f_{k,1}$, from the proof of Lemma \ref{k23res} up to the Equation \eqref{hk232}, the proposition follows.
\end{proof}
\begin{remark}
    Let $k\in \N, k>1$. Using the explicit form of $h_{k,2}$ polynomials, we can find the exact value of the resultant of $h_{k,2}$ and $b-a+i$ to be, 
    \begin{center}
            $\mbox{Res}(h_{k,2},b-a+i)= \left\{ \begin{array}{rcl}
                3(2ai+1); & k \mbox{ even,} & k>0, \\ -3ai; & k \mbox{ odd,} & k>1. \\
            \end{array}\right.$
        \end{center}
    We note that introducing the explicit form of $h_{k,2}$ in the proof of irreducibility of $\mathscr{S}_{k,2}$ above, is unnecessary and does not make the proof any shorter. Therefore, we have kept it as a separate section. 
\end{remark}
\section{The unicritical case}\label{unicritical}
    Putting $a=0$ in Equation \eqref{newmaineq}, we get the normal form for monic, reduced, unicritical cubic polynomial, $$f(z)= z^{3} +b.$$ Let $R_{k,n}$ be the polynomial in $\Z[b]$, whose roots are exactly the values of $b$ for which $0$ is strictly $(k,n)$-preperiodic under $f(z)= z^{3}+b$. Putting $a=0$ in Equation \eqref{hkn}, we get that $R_{k,n}$ divides $h_{k,n}(0,b)$ in $\Z[b]$. In \cite{Milnor+2014+15+24}, Milnor conjectured that,
    \begin{conjecture}
        The polynomial $R_{k,n}$ is either constant or irreducible over $\Q$, for any $k\geq 0, n\geq 1$.
    \end{conjecture}
    In this section, we prove the irreducibility of $R_{k,2}$ over $\Q$, for any $k \geq 0,\; k$ even.
    \begin{remark}\label{coales}
        The polynomial $R_{k,n}$ can be constant for some $(k,n)\in \Z^{2}, k\geq 0, n\geq 1$. For example, $R_{1,n}$ is equal to one for any $n\in \N$. This can be shown from the following observation: if the totally ramified critical point $0$ is $(1,n)$-preperiodic for the polynomial $f(z)=z^{3} + b$, then $0$ is $n$-periodic too.
    \end{remark}
    \begin{theorem}\label{uni}
        For any even $k\geq 0$, the polynomial $R_{k,2}$ is either constant or an irreducible polynomial over $\Q$. 
    \end{theorem}
    \begin{proof}
         Case $k=0:$ From Lemma \ref{h023}, the polynomial $h_{0,2}$ is $(b-a)(b+2a)+1$. As $h_{0,2}(0,b)= b^{2}+1$ is irreducible over $\Q$, the polynomial $R_{0,2}$ is same as $h_{0,2}(0,b)$ and it is irreducible over $\Q$.
         
         Case $k\in \N, k$ even: From section \ref{t''}, the polynomial $h_{k,2}$ is $3$-Eisenstein with respect to the polynomial $h_{1,2}=(b-a)^{2}+1$. From the proof of Corollary \ref{k23res2}, we get that for any even $k\in \N$, the resultant $$\mbox{Res}(h_{k,2},h_{1,2})= 9(4a^{2}+1).$$ For any $k\in \N$, $\;h_{k,2}$ is monic as polynomial in $b$ over the integral domain $\Z[a]$. Hence, degree of $h_{k,2}$ as a polynomial in $b$ over $\Z[a]$ is same as the degree of $h_{k,2}(0,b)$ as an element of $\Z[b]$. By invariance of resultant under ring homomorphisms that preserve the degree of the polynomials, we see that $$\mbox{Res}(h_{k,2}(0,b), h_{1,2}(0,b))= \mbox{Res}(h_{k,2}, h_{1,2}) (0) =9.$$  Hence, the polynomial $h_{k,2}(0,b)$ is $3$-Eisenstein with respect to the polynomial, $h_{1,2}(0,b)= b^{2}+1$. So for even $k$, the polynomial $h_{k,2}(0,b)$ is irreducible over $\Q$. As $R_{k,2}$ divides $h_{k,2}(0,b)$, we get that $R_{k,2}$ is either constant or an irreducible polynomial in $\Q[b]$.
    \end{proof}
    \begin{remark}
        For $k>1, k$ odd, the resultant $\mbox{Res}(h_{k,2}, h_{1,2})=9a^{2}$ (from the proof of Lemma \ref{k23res2}). So, $$\mbox{Res}(h_{k,2}(0,b), h_{1,2}(0,b))= \mbox{Res}(h_{k,2}, h_{1,2}) (0) =0.$$ This means that $b^{2}+1$ divides $h_{k,2}(0,b)$, for $k$ odd, $k>1$. This restricts a generalization of proof of Theorem \ref{uni} to the case, $k$ odd, $k>1$. 
    \end{remark}
    \begin{remark}
        Theorem \ref{uni} partially proves Milnor's conjecture (\cite{Milnor+2014+15+24}) on the unicritical case. A stronger version of this theorem has been proved in \cite{Gok20}. A parallel version of this theorem that arises from replacing the normal form of unicritical cubic polynomial $f(z)= z^{3}+b$ with the normal form $g(z)= cz^{3}+1$, has been proved by Buff, Epstein and Koch in \cite{article1}.
    \end{remark}
\section{On $\mathscr{S}_{k,q}$ curves}\label{kqcurves}
In this section, we study the obstructions to direct generalization of our technique to study irreducibility of $\mathscr{S}_{k,q}$ curves, where $k\geq 0, q$ odd prime.
    \begin{lemma}
        For any prime $q\in \N$, we have $h_{1,q}=(f^{q}(a)+2a)/(b+2a)$. In another form, $h_{1,q}\equiv h_{0,q} \equiv \sum_{i=0}^{q-1}(b-a)^{3^{i}-1} (\text{mod }3)$.
    \end{lemma}
    \begin{proof}
        We deduce,
        \begin{align*}
            f_{1,q} & =f^{q+1}(a)-f(a)\\
            & =(f^{q}(a)-a) \left((f^{q}(a))^{2}+af^{q}(a)+a^{2}-3a^{2}\right)\\
            & =(f^{q}(a)-a)^{2} (f^{q}(a)+2a)\\ 
            & = (f_{0,q})^{2} (f^{q}(a)+2a).
        \end{align*}
        So, $h_{1,q}$ divides $f^{q}(a)+2a$. As $f^{q}(a)+2a\equiv 3a\;(\text{mod }f_{0,q})$, the polynomials$f^{q}(a)+2a$  and $f_{0,q}$ are coprime.

        As we obtain $h_{1,q}$ by factoring out irreducible factors of $f_{1,1}$ and $f_{0,q}$ from $f_{1,q}$, each raised to their highest power that divides $f_{1,q}$, we get that $$h_{1,q}= \frac{f^{q}(a)+2a}{(h_{1,1})^{s}},$$ where $s$ is the highest power of $h_{1,1}$ that divides $f^{q}(a)+2a$.

        By Equation \eqref{f113}, we have $h_{1,1}=b+2a$. Putting $b=-2a$, we see that $f^{n}(a)=-2a, $ for any $ n\in \N$. So, $b+2a$ divides $f^{q}(a)+2a$. We need to check if $(b+2a)^{2}$ divides $f^{q}(a)+2a$. As $f\equiv z^{3}-a^{3} +b \;(\text{mod } 3)$, we have $$f^{q}(a)+2a \equiv f^{q}(a)-a\equiv \sum_{i=0}^{p-1} (b-a)^{3^{i}}\;(\text{mod }3).$$ As $(b-a)^{2}$ does not divide $f^{q}(a)+2a$ in modulo $3$, we get that $(b+2a)^{2}$ does not divide $f^{q}(a)+2a$ in $\Z[a,b]$. Hence, $$h_{1,q}=\frac{f^{q}(a)+2a}{b+2a}.$$ Reducing this equation in modulo $3$, we obtain the other form of $h_{1,q}$ mentioned in the lemma. To show $h_{1,q}\equiv h_{0,q}(\text{mod }3)$, observe that $$f_{0,q}\equiv (b-a) \sum_{i=0}^{q-1} (b-a)^{3^{i}-1}\;(\text{mod }3).$$ Hence, the lemma is proved. 
    \end{proof}
    \begin{lemma}
        Let $k\in \N $ and $q\in \N, q$ prime. If $h_{1,q}$ is irreducible in $\F_{3}[a,b]$, then $h_{k,q}\equiv (h_{1,q})^{N_{k,q}}(\text{mod }3)$, for some $N_{k,q}\in \N$.
    \end{lemma}
    \begin{proof}
        Let $k,q\in \N, q$ prime. Then,
        \begin{align*}
            f_{k,q} & = f^{k+q}(a)- f^{k}(a)\\
            & \equiv \sum_{i=k}^{k+q-1} (b-a)^{3^{i}} (\text{mod }3)\\
            & \equiv (b-a)^{3^{k}} \left(\sum_{j=0}^{q-1} (b-a)^{3^{j}-1} \right)^{3^{k}} (\text{mod }3)
        \end{align*}
        As $f_{k,1}\equiv (b-a)^{3^{k}}(\text{mod } 3)$, we have $h_{k,q}$ divides $\left(\sum_{j=0}^{q-1} (b-a)^{3^{j}-1}\right)^{3^{k}}$ in modulo $3$. As $h_{1,q}\equiv \sum_{j=0}^{q-1} (b-a)^{3^{j}-1}( \text{mod }3)$ is irreducible modulo $3$, the lemma follows.
    \end{proof}
    The next lemma shows that this method of showing irreducibility of $h_{k,q}$ in $\Q[a,b]$, does not extend for any prime $q$ other than $2$.
    \begin{lemma}
        The polynomial $h_{1,q}(\text{mod }3)$ is irreducible in $\F_{3}[a,b] \iff q=2$.
    \end{lemma}
    \begin{proof}
        By $\tilde{h}_{1,q},$ we will denote the image of $h_{1,q}$ under the quotient map $\Z[a,b] \rightarrow \F_{3}[a,b]$. The polynomial $\tilde{h}_{1,q}= \sum_{i=0}^{q-1}(b-a)^{3^{i}-1}$ is reducible in $\F_{3}[a,b]$ iff $g(x):=\sum_{i=0}^{q-1}x^{3^{i}-1}$ is reducible in $\F_{3}[x]$. Consider the polynomial $xg(x)= \sum_{i=0}^{q-1}x^{3^{i}}$. Consider the extension $\F_{3^{q}}$ over $\F_{3}$. The orders of the two fields in this extension imply that there are non-zero elements in $\F_{3^{q}}$, for which the trace is $0$ under the extension $\F_{3^{q}}/\F_{3}$. Any such element is a root of $xg(x)$, as $Gal(\F_{3^{q}}/\F_{3})$ is generated by the Frobenius elemnent, $x\mapsto x^{3}$. So, $g(x)$ is irreducible in $\F_{3}[x]$ if and only if $deg(g(x))$ is same as 
        $[\F_{3^{q}}:\F_{3}]=q$. Now, $deg(g(x))=3^{q-1}-1$ is equal to $q\iff q=2$. Hence, the lemma is proved.  
    \end{proof}
\section*{Acknowledgement}
    The author would like to thank C. S. Rajan, Sagar Shrivastava and Manodeep Raha, for their valuable inputs and many discussions. The author also expresses his gratitude to Ashoka University for providing a productive workspace, where this work was done. The author is also thankful to Kerala School of Mathematics for his visit there, where he had many helpful discussions with Plawan Das, Subham Sarkar, and M. M. Radhika.

\end{document}